\documentclass[reqno]{amsart}
\usepackage{amssymb}
\usepackage{amsthm}
\usepackage{amsmath}
  \usepackage{paralist}
  \usepackage{graphics,color}
  \usepackage{epsfig}
\usepackage[english]{babel}

\newtheorem{theorem}{Theorem}[section]

\newtheorem{lemma}[theorem]{Lemma}

\newtheorem{hyp}[theorem]{Hypotheses}

\theoremstyle{definition}

\newtheorem{remark}[theorem]{Remark}

\newcommand{\rl}{{{\rm I} \kern -.15em {\rm R} }}

\newcommand{\rsp}{{\mathbb R}}

\newcommand{\ds}{\displaystyle}
\newcommand{\no}{\nonumber}

\newcommand{\beq}{\begin{equation}}
\newcommand{\eeq}{\end{equation}}
\newcommand{\beqn}{\begin{eqnarray}}
\newcommand{\eeqn}{\end{eqnarray}}
\newcommand{\bear}{\begin{array}}
\newcommand{\eear}{\end{array}}
\newcommand{\beit}{\begin{itemize}}
\newcommand{\eeit}{\end{itemize}}
\newcommand{\beqno}{\begin{eqnarray*}}
\newcommand{\eeqno}{\end{eqnarray*}}

\newcommand{\ve}{\varepsilon}
\newcommand{\g}{\gamma}
\renewcommand{\l}{\lambda}
\renewcommand{\phi}{\varphi}
\newcommand{\dl}{\delta}
\newcommand{\s}{\sigma}

\newcommand{\G}{\Gamma}

\newcommand{\om}{\omega}
\newcommand{\Om}{\Omega}
\renewcommand{\a}{\alpha}

\newcommand{\sub}{\subset}
\newcommand{\ov}{\overline}
\newcommand{\what}{\widehat}

\newcommand{\wtil}{\widetilde}
\newcommand{\pn}{\par \noindent}
\newcommand{\med}{\medskip}
\newcommand{\qq}{\qquad}
\newcommand{\q}{\quad}

\newcommand{\media}[1]
{\kern 0.4ex {-} \kern -2.0 ex {\int}_{\kern -0.9 ex {#1}}\;}

\allowdisplaybreaks

\begin{document}

\title[A severely ill-posed problem]{Continuous dependence and
uniqueness\\
for lateral Cauchy problems \\
for linear\\ integro-differential parabolic equations\\
}

\author[A. Lorenzi, L. Lorenzi and M. Yamamoto]
{Alfredo Lorenzi, Luca Lorenzi
and Masahiro Yamamoto}

\address{L.L.: Dipartimento di Matematica e Informatica, Universit\`a degli Studi
di Parma,  Parco Area delle Scienze 53/A, I-43124 Parma, Italy.}
\email{luca.lorenzi@unipr.it}

\address{M.Y.: Department of Mathematical Sciences, The University of Tokyo,
Komaba, Meguro, Tokyo 153, Japan}
\email{myama@ms.u-tokyo.ac.jp}
\thanks{This paper was completed after the first author passed away.}

%

\begin{abstract}
Via Carleman estimates we prove uniqueness and continuous
dependence results for lateral Cauchy problems for linear integro-differential
parabolic equations without initial conditions. The additional information
supplied prescribes the conormal derivative of the temperature on a relatively open
subset of the lateral boundary of the space-time domain.
\end{abstract}



\maketitle

\section{Introduction}
\setcounter{equation}{0}
In this paper we consider the linear ill-posed integro-differential
parabolic problem with {\it no initial condition}
\begin{align}
\left\{
\begin{array}{ll}
D_tu(t,x)-A(x,D)u(t,x)
= {\mathcal B}u(t,x)+f_0(t,x), \q & (t,x)\in (0,T)\times \Om,\\[1mm]
u(t,x)=g(t,x), & (t,x)\in (0,T)\times \partial\Om,\\[1mm]
D_{\nu_A}u(t,x)=D_{\nu_A}g(t,x), & (t,x)\in (0,T)\times \G.
\end{array}
\right.
\label{1}
\end{align}
Here $\Om$ is a bounded connected open set in ${\mathbb R}^n$ whose boundary
$\partial \Om$ is of $C^2$-class, $\Gamma \subset \partial\Omega$
is a sub-domain of $\G$, i.e., a relatively open subset of $\partial\Omega$.
Moreover,
\begin{equation}
A(x,D)=\sum_{i,j=1}^nD_{x_i}(a_{i,j}(x)D_{x_j})+\sum_{j=1}^nb_j(x)D_{x_j}+a_0(x)
\label{oper-A}
\end{equation}
is an elliptic operator which generates an analytic semigroup $\{e^{tA}\}_{t\ge 0}$ in $L^2(\Om)$.
The operator ${\mathcal B}$ is defined by
\begin{align}
{\mathcal B}u(t,x)=& f_1(t,x)u(T_1,x)+f_2(t,x)u(T_2,x)+f_3(t,x)
\int_{T_1}^{T_2} \rho_1(\s,x)u(\s,x)\,d\s\no\\
&+ Bu(t,x) + f_4(t,x)\int_{T_1}^{T_2}\rho_2(\s,x)Bu(\s,x)\,d\s
=:\sum_{j=1}^5\,{\mathcal B}_ju(t,x),
\label{2.12}
\end{align}
where $0<T_1<T_2<T$ and
\beqn
\label{2.5}
Bu(t,x)=\int_\Om k(t,x,y)u(t,y)\,dy,
\eeqn
the kernel $k:(0,T)\times \Om\times \Om\to \rsp$ being a
measurable function.
The functions $f_0, f_1, f_2, f_3, f_4, \rho_1,\rho_2, k, g$ are suitably chosen
so as to satisfy Hypotheses \ref{hyp1} and \ref{hyp-1a} stated in Sections \ref{sect2} and \ref{sect3}.
Finally, $\nu_A$ denotes the conormal vector related to the operator $A(x,D)$, i.e., $(\nu_A(x))_i=
\sum_{j=1}^na_{i,j}(x)\nu_j(x)$ for any $i=1,\ldots,n$ and $x\in\partial\Omega$, where $\nu$ denotes the outward unit normal vector to $\partial\Omega$ at $x$.

We consider the inverse problem of determining $u$ by the knowledge of
$f_0$ and $g$.
Our main results are the uniqueness: $f_0=0$ and $g=0$ imply $u=0$ in $(0,T)
\times \Omega$  and the continuous dependence of
$u$ in terms of $(f_0,g)$.
Continuous dependence means here that $u$ is estimated in $C((0,T];L^2(\Om))\cap L^2_{\rm loc}((0,T];H^1(\Om))$ in terms
of the $H^1(0,T;L^2(\Omega))$-norm of $f_0$ and the $L^2(0,T;H^2(\Om))$-norm of $g$.

When the non-local term ${\mathcal B}$ is not included, that is, when
we have to deal with a differential problem,
we can apply the Carleman estimate in \cite{I} (see also \cite{FI}) and
prove the uniqueness and
continuous dependence. With the presence of ${\mathcal B}$, to the best
knowledge of the authors, results are not available in literature.

Our method is still based on the Carleman estimate in \cite{I}, but in order to
treat the non-local terms, we need strong conditions on the kernels
$\rho_1$, $\rho_2$ and $k$ in ${\mathcal B}$.

Finally, we stress that, due to the absence of initial conditions, our results can concern both forward and backward
parabolic problems.

Carleman estimates are a powerful tool in solving inverse problems. We
refer the readers to the pioneering work \cite{BK} and also to \cite{K} and the survey \cite{Y} related to parabolic inverse problems.
Concerning uniqueness and continuous dependence results for Cauchy problems with no initial conditions, we mention the papers \cite{LO,LL-1,LL-2}. More specifically, in \cite{LO}, $f_1=f_2=f_3=f_4=0$ and only the
Dirichlet boundary condition is prescribed on $\partial \Om$.
Two different additional conditions are assumed. In the first case, $u$ is assumed to be known in an open subdomain $\om$
with ${\ov \om}\sub \Om$, while in the latter the linear operator
${\mathcal B}$, {\it which transforms spatial arguments}, is defined by
\begin{eqnarray*}
{\mathcal B}u(t,x)=k_0(t,x)u(t,\sigma x)+\sum_{j=1}^n\,k_j(t,x)D_{x_j}
u(t,\rho x),
\end{eqnarray*}
for some $\sigma\in (0,1)$, where $\Om$ is convex with respect to $x=0$.
In \cite{LL-1} the case when the elliptic operator $A(\cdot,D)$ has smooth and unbounded coefficients in a cylinder of $\mathbb R^{m+n}$ and it degenerates on some directions
is considered and new Carleman estimates are proved. Finally, in \cite{LL-2} problem \eqref{1} is considered with Dirichlet boundary conditions on $\partial\Omega$
and first order additional conditions on a part of $\partial\Omega$. Also in this situation, new Carleman type estimates have been the key tool to prove the
uniqueness and continuous dependence results.

We conclude this introduction with giving the plan of the paper. In Section \ref{sect2} we
state the problems that we deal with in the paper and introduce the well-known Carleman estimates for
linear parabolic operators (e.g., \cite{FI}-\cite{IY}).
In Section \ref{sect3} we establish the uniqueness result (Theorem 3.1) for our problem
and prove it.  The proof is based on the Carleman estimate.
Finally, Section \ref{sect4} is devoted to deducing the continuous dependence result
in non-weighted $L^2$-spaces (Theorem \ref{teo-4.1}).

\subsection*{Notation}
Throughout the paper we set
$Q_{T_1,T_2} = (T_1,T_2) \times \Omega$ for any $T_1, T_2\in\mathbb R$ with $T_1<T_2$ and we simply
write $Q_T$ for $Q_{0,T}$.
\newpage

\section{Main assumptions and preliminary results}
\label{sect2}
\setcounter{equation}{0}
To begin with, let us introduce our standing assumptions. For this purpose, we introduce the function
$l:[0,T]\to\mathbb R$, defined by $l(t)=t(T-t)$ for any $t\in [0,T]$, and a
function $\psi\in C^2({\ov \Om})$ which satisfies
the following properties:
\begin{align*}
&\psi(x)>0,\ x\in \Om,\qq\;\, |\nabla\psi(x)|>0,\
x\in {\ov \Om},\\
&D_{\nu_A}\psi(x):=\sum_{i,j=1}^na_{i,j}(x)\nu_j(x)D_i\psi(x)\le 0,\
x\in \partial\Om\setminus \G.
\end{align*}
For the existence of such a function, we refer the reader to \cite{FI}.

\begin{hyp}
\label{hyp1}
\begin{enumerate}[\rm (i)]
\item
$\Om$ is a bounded open set in ${\mathbb R}^n$, $\partial \Om$
being of $C^2$-class;
\item
$\Gamma \subset \partial\Omega$ is an arbitrarily fixed sub-domain of $\Gamma$;
\item
the coefficients of the operator $A(x,D)$, defined in \eqref{oper-A}, satisfy the following conditions:
\begin{enumerate}[\rm (a)]
\item
$a_{i,j}\in C^2({\ov \Om})$, $a_{j,i}=a_{i,j}$, for any $i,j=1, \ldots, n$;
\item
$a_j\in C^1(\ov\Om)$ for any $j=0,\ldots, n$,
\item
$a_0\in C(\overline\Omega)$;
\item
$\sum_{i,j=1}^n\,a_{i,j}(x)\xi_i\xi_j \ge \mu_0|\xi|^2$ for any $x\in {\ov \Om}$, $\xi\in \rsp^n$ and some positive constant $\mu_0$;
\end{enumerate}
\item
$f_0\in L^2((0,T)\times \Omega)$ and $g\in H^1(0,T;L^2(\Om))\cap L^2(0,T;H^2(\Om))$;
\item
$f_1, f_2, f_3,f_4\in L^2(0,T;L^\infty(\Om))$;
\item
$\rho_1, \rho_2$ belong to $L^2((0,T);L^{\infty}(\Omega))$.
\item
$k$ is a measurable function in $(0,T)\times\Omega\times\Omega$. Moreover,
the functions $(t,y)\mapsto (l(t))^{3-\gamma}\|k(t,\cdot,y)\|_{L^1(\Omega)}$
and $(t,x)\mapsto (l(t))^{3-\gamma}\|k(t,x,\cdot)\|_{L^1(\Omega)}$ are bounded in $Q_T$.
\end{enumerate}
\end{hyp}

\begin{remark}
The conditions on $\rho_1$, $\rho_2$ and $k$ will be refined in Section \ref{sect3}.
\end{remark}

In this paper, our first main problem is:\\
(IP1):
{\it estimate the solution $u$
in $C((0,T);L^2(\Om))\cap L^2_{\rm loc}((0,T];H^1(\Om))$ to the problem}
\beqno
(IP1)\ \left\{ \hskip-2truemm
\bear{ll}
u\in H^1(0,T;L^2(\Om))\cap L^2(0,T;H^2(\Om)),
\\[2mm]
D_tu(t,x)-A(x,D)u(t,x)= {\mathcal B}u(t,x)+f_0(t,x),\q & (t,x)\in (0,T)\times \Om,
\\[2mm]
u(t,x)=g(t,x),& (t,x)\in (0,T)\times \partial\Om,
\\[2mm]
D_{\nu_A}u(t,x)=D_{\nu_A}g(t,x),& (t,x)\in (0,T)\times \G,
\eear \right.
\eeqno
where the linear operator ${\mathcal B}$ is defined by \eqref{2.12}
and \eqref{2.5}.
\med

We can consider another problem:\\
(IP1$'$) {\it estimate in
$C((0,T);L^2(\Om))\cap L^2_{\rm loc}([0,T);H^1(\Om))$ the solution $u$ to the problem}
\beqno
(IP1')\ \left\{ \hskip-2truemm
\bear{ll}
u\in H^1(0,T;L^2(\Om))\cap L^2(0,T;H^2(\Om)),
\\[2mm]
D_tu(t,x)+A(x,D)u(t,x)
= {\mathcal B}u(t,x)+f_0(t,x), \q & (t,x)\in (0,T)\times \Om,
\\[2mm]
u(t,x)=g(t,x),& (t,x)\in (0,T)\times \partial\Om,
\\[2mm]
D_{\nu_A}u(t,x)=D_{\nu_A}g(t,x),& (t,x)\in (0,T)\times \G.
\eear \right.
\eeqno
By the change of the unknown function
$w(t,x)=u(T-t,x)$ for $(t,x)\in (0,T)\times \Om$,
the problem (IP1$'$) changes to problem (IP1) with $({\mathcal B},f_0,g)$
being replaced by $(\widehat {\mathcal B},\widehat f_0,-\widehat g)$, where
${\what h}(t,x)=-h(T-t,x)$ for a given function $h$
and the linear operator ${\what {\mathcal B}}$ is defined by
\begin{align*}
{\what {\mathcal B}}w(t,x)=&{\what f}_1(t,x)w({\what T}_2,x)
+{\what f}_2(t,x)w({\what T}_1,x)-{\what f}_3(t,x)\int_{{\what T}_1}^{{\what T}_2}
{\what \rho_1}(\s,x)w(\s,x)\,d\s\no\\
&+ {\what B}w(t,x)
+ {\what f}_4(t,x)\int_{{\what T}_1}^{{\what T}_2} {\what \rho}_2(\s,x)
{\what B}w(\s,x)\,d\s
=:\sum_{j=1}^5\,{\what {\mathcal B}}_ju(t,x),
\end{align*}
where ${\what T}_1=T-T_2$, ${\what T}_2=T-T_1$ and
\beqno
{\what B}w(t,x)=\int_\Om {\what k}(t,x,y)w(t,y)\,dy.
\eeqno
Thus (IP1$'$) is led back to the problem (IP1), which is a
forward problem in time.
\med

\begin{remark}
It is a simply task to check that,
if Hypotheses \ref{hyp1} hold true for
$(f_0,f_1,f_2,f_3,f_4,\rho_1,\rho_2,k)$, then they hold true also for
$({\what f}_0,{\what f}_1,{\what f}_2,{\what f}_3,{\what f}_4,
{\what \rho}_0,{\what \rho}_1,{\what k})$. We stress that, if
the triplet $(\rho_1,\rho_2,k)$ satisfies the forthcoming Hypotheses \ref{hyp-1a},
then the triplet $({\what \rho}_1,{\what \rho}_2,{\what k})$ satisfies the same conditions with
the {\it same constants}, since $l(T-t)=l(t)$ for any $t\in [0,T]$.
\end{remark}

Coming back to problem (IP1) and
introducing the function
$v = u-g$,
where $u$ is the solution to problem (IP1), we can reduce (IP1)
to the problem with homogeneous boundary condition:
\beqno
(IP2)\ \left\{ \hskip-2truemm
\bear{ll}
v\in H^1(0,T;L^2(\Om))\cap L^2(0,T;H^2(\Om)),
\\[2mm]
D_tv(t,x)-A(x,D)v(t,x) = {\mathcal B}v(t,x)+{\wtil f}(t,x),
\q & (t,x)\in (0,T)\times \Om,
\\[2mm]
v(t,x)=0, & (t,x)\in (0,T)\times \partial \Om,
\\[2mm]
D_{\nu_A}v(t,x)=0, & (t,x)\in (0,T)\times \G,
\eear \right.
\eeqno
where
\beqn
\label{2.17}
{\wtil f}=f_0-D_tg+A(\cdot,D)g+{\mathcal B}g.
\eeqn
Therefore, we mainly consider problem $(IP2)$.

Now we state a key Carleman estimate.
For this purpose, we introduce the functions $\varphi_\l:{\ov \Om}\to \rsp$ and
$\a_\l:[0,T]\times {\ov \Om}\to \rsp$ with
$\l\in [1,+\infty)$, defined by
\beqn
\varphi_\l(x)=e^{\l \psi(x)},\qq\;\,
\a_\l(t,x)=\frac{e^{\l \psi(x)}-e^{2\l \|\psi\|_\infty}}{l(t)},
\q t\in (0,T)\;\,x\in\ov \Om.
\label{alpha}
\eeqn
By \cite[Lemma 2.4]{IY} (see also \cite{FI,I}) and
since $\varphi_\l(x)\ge 1$ for all $x\in {\ov \Om}$,
there exists $\widehat{\lambda}$ such that for any $\lambda\ge\widehat{\lambda}$ we can choose $\widehat{s}_0=\widehat{s}_0(\lambda) > 0$
and $C_1=C_1(\lambda)>0$ such that the following Carleman estimate
\begin{align}
&s^3\int_{Q_T} (l(t))^{-3}|v(t,x)|^2\exp[2s\a_\l(t,x)]\,dtdx\no\\
&+ s\int_{Q_T} (l(t))^{-1}|\nabla_xv(t,x)|^2\exp{[2s\a_\l(t,x)]}\,dtdx\no\\
&+s^{-1}e^{-\l\|\psi\|_{\infty}}\int_{Q_T}l(t)\bigg[|D_tv(t,x)|^2+\sum_{i,j=1}^n\,|D_{x_i}D_{x_j}v(t,x)|^2\bigg ]
\exp{[2s\a_\l(t,x)]}\,dtdx\no\\
\le &s^3\int_{Q_T} (l(t))^{-3}(\varphi_\l(x))^3|v(t,x)|^2\exp{[2s\a_\l(t,x)]}\,dtdx\no\\
&+ s\int_{Q_T}(l(t))^{-1}\varphi_\l(x)|\nabla_xv(t,x)|^2
\exp{[2s\a_\l(t,x)]}\,dtdx\no\\
&+s^{-1}\int_{Q_T}\!l(t)(\varphi_\l(x))^{-1}\bigg[|D_tv(t,x)|^2\!+\!\sum_{i,j=1}^n|D_{x_i}D_{x_j}v(t,x)|^2\bigg ]\!
\exp{[2s\a_\l(t,x)]}\,dtdx
\no\\
\le &6C_1\int_{Q_T}\bigg [|{\wtil f}(t,x)|^2+\sum_{j=1}^5\,|{\mathcal B}_jv(t,x)|^2\bigg ]\exp{[2s\a_\l(t,x)]}\,dtdx
\label{2.21}
\end{align}
is satisfied by all $s >\widehat{s}_0$ and any solution $v\in H^1(0,T;L^2(\Om))\cap
L^2(0,T;H^2(\Om)\cap H^1_0(\Om))$ to problem (IP2).
Moreover, the positive constants $C_1$, $\l$ and ${\what s}_0$
depend on $\mu_0$, $\mu_1$, $T$, $\|a_{i,j}\|_{L^\infty(\Om)}$,
$\|a_j\|_{L^\infty(\Om)}$, $\|a_0\|_{L^\infty(\Om)}$,
$i,j=1,\ldots,n$, $\Om$ and $\G$.

\begin{remark}
{\rm Note that the Carleman estimate in \cite[Lemma 2.4]{Y} actually contains the $L^2$-norms of $e^{s_0\alpha_{\lambda}}v$, $e^{s_0\alpha_{\lambda}}D_tv$ and
$e^{s_0\alpha_{\lambda}}D_{x_j}v$ ($j=1,\ldots,n$) on $(0,T)\times\Gamma$ on its right-hand side. In our situation all these terms
identically vanish on $(0,T)\times\partial\Omega$. Indeed, since $v=0$ almost everywhere on $(0,T)\times\Omega$, $D_tv$ and the tangential spatial derivatives of $v$
vanish almost everywhere on $(0,T)\times\partial\Omega$ as well. On the other hand, since the conormal derivative
of $v$ vanishes on $(0,T)\times\Gamma$ and for any $x\in\Gamma$ we can split an arbitrary vector of $\mathbb R^n$ along
$\nu_A(x)$ and the tangential directions, we conclude that $\nabla v$ vanishes almost everywhere on $(0,T)\times\Gamma$.}
\end{remark}

\section{Uniqueness result}
\label{sect3}
\setcounter{equation}{0}
In this section $\lambda\ge\what\lambda$ is fixed and for notational convenience we set
\begin{eqnarray*}
c_{1,\lambda}(\psi):= e^{2\l \|\psi\|_\infty}-e^{\l\psi_m}
\end{eqnarray*}
where $\psi_m$ denotes the minimum of the function $\psi$.

We also assume the following additional set of assumptions.

\begin{hyp}
\label{hyp-1a}
There exist five positive constants $K_j$ $(j=1,\ldots,5)$ and $0<T_1<T_2<T$ such that
\beqn
&&\label{3.12}
|\rho_1(t,x)|\le K_1\exp [s_0\alpha_{\lambda}(t,x)],\qquad\;\,(t,x)\in Q_T,
\\[2mm]
\hskip-15truemm
&&\label{3.28}
|\rho_2(t,x)|\le K_2\exp [s_0\alpha_{\lambda}(t,x)],\qquad\;\,(t,x)\in Q_{T_1,T_2},
\eeqn
Moreover,
\begin{equation}
\label{3.0d}
K_3:={\rm ess\,sup}_{(t,x)\in Q_T}\, (l(t))^\g\int_\Om |k(t,x,y)|\,dy
<+\infty,
\end{equation}
for some $\gamma\in [0,3]$ and
\begin{align}
&\;\;\;\;\;\;\int_{\{x\in \Om:\ \psi(x)>\psi(y)\}} |k(t,x,y)|\,dx
\le K_4(l(t))^{\g-3}\exp{[-2s_0c_{1,\lambda}(\psi)(l(t))^{-1}]},\label{3.15}
\\[1mm]
&\;\;\;\;\;\;\int_{_{\{x\in \Om:\ \psi(x)\le\psi(y)\}}}
|k(t,x,y)|\,dx\le K_5(l(t))^{\gamma-3},
\label{3.16}
\end{align}
for any $(t,y)\in Q_T$.
\end{hyp}

\begin{remark}
We stress that the condition \eqref{3.12} implies that the kernel
$\rho_1$ should exponentially decay to $0$ at $t=0$ and $t=T$.
\end{remark}

Next, we choose $s_0\ge\what s_0$ so as to satisfy the inequalities
\begin{align}
H_0(s_0):=&6C_1\Big\{\Big(2^{-6}T^6[(T_2\!-\!T_1)^{-1}\!+\!s_0^{1+\dl}]
+2^{-2}T^{3}s_0c_{1,\lambda}(\psi)\Big )\sum_{j=0}^1\|f_j\|^2_{L^2(0,T;L^{\infty}(\Omega))}\no\\[1mm]
&\phantom{6C_1\Big\{ } + 2^{-6}T^6(T_2-T_1)K_2^2\|f_{3}\|^2_{L^2(0,T;L^\infty(\Om))}
+K_3(K_4+K_5)\no\\[1mm]
&\phantom{6C_1\Big\{} +(T_2-T_1)K_1^2K_3(K_4+K_5)\|f_{4}\|^2_{L^2(0,T;L^\infty(\Om))}\Big\}
\le \frac{1}{2}s_0^3,\label{3.0b}
\\[2mm]
H_1(s_0):=&6C_1M_{T_1,T_2}^{-1}s_0^{-(1+\dl)}
\sum_{j=0}^1\|f_j\|^2_{L^\infty(\Om;L^2(0,T))}\le \frac{1}{2}s_0^{-1}e^{-\lambda\|\psi\|_{\infty}},
\label{3.0a}
\end{align}
$C_1$ being the positive constant in estimate \eqref{2.21},
$K_3, K_4, K_5$ being given in \eqref{3.0d}-\eqref{3.16} and $M_{T_1,T_2}=[\min\{T_1(T-T_1),T_2(T-T_2)\}]^{-1}$.
Observe then that, for all $(t,x)\in Q_T$, we have
\beqn
\label{3.8}
\exp[-2s_0c_{1,\lambda}(\psi)(l(t))^{-1}] \le \exp[2s_0\a_\l(t,x)]
\le 1,
\eeqn

Then we show our first main result.
\begin{theorem}
\label{teo-1}
Let Hypotheses $\ref{hyp1}$, $\ref{hyp-1a}$ and conditions \eqref{3.0b}, \eqref{3.0a} be satisfied. Further, let $u$
be a strong solution to problem $(IP1)$. Then, the following weighted estimate
\begin{align}
&\frac{1}{2}s_0^3\int_{Q_T}(l(t))^{-3}|v(t,x)|^2\exp{[2s_0\a_\l(t,x)]}\,dtdx\no\\
&+s_0\int_{Q_T}(l(t))^{-1}|\nabla_xv(t,x)|^2\exp{[2s_0\a_\l(t,x)]}\,dtdx
\no\\
& + \frac{1}{2}s_0^{-1}e^{-\lambda\|\psi\|_{\infty}}\int_{Q_T} l(t)|D_tv(t,x)|^2\exp{[2s_0\a_\l(t_j,x)]}\,dtdx
\no\\
\le & 6C_1\int_{Q_T} |{\wtil f}(t,x)|^2\exp{[2s_0\a_\l(t,x)]}\,dtdx,
\label{3.25}
\end{align}
holds true with $v=u-g$ and
$s\ge {\what s}_0$. In particular, problem $(IP1)$ admits
at most one solution.
\end{theorem}


The rest of this section is devoted to the proof of Theorem 3.1.

\subsection{Estimating ${\mathcal B}_1$ and ${\mathcal B}_2$}
First we need some weighted trace results.



\begin{lemma}
\label{lem-3.2}
The following estimate holds true for all
$w\in H^1(T_1,T_2;L^2(\Om))$, $r_0\ge 0$, $\ve>0$ and $j=1,2$:
\begin{align}
&\int_{\Om}|w(T_j,x)|^2\exp{[2r_0\a_\l(T_j,x)]}\,dx\notag\\
\le &\ve^2\int_{Q_{T_1,T_2}}|D_tw(t,x)|^2\exp{[2r_0\a_\l(t,x)]}\,dtdx
\no\\
& + \int_{Q_{T_1,T_2}} |w(t,x)|^2\{(T_2-T_1)^{-1}+\ve^{-2}
+2r_0c_{1,\lambda}(\psi)|l'(t)|(l(t))^{-2}\}\notag\\
&\qquad\qquad\qquad\times
\exp{[2r_0\a_\l(t,x)]}\,dtdx.
\label{3.4b}
\end{align}
\end{lemma}

\begin{proof}
By a density argument, we can assume that $w$ is smooth enough.
We arbitrary fix $x\in\Omega$. From the identity
\begin{align*}
&|w(t,x)|^2\exp{[2r_0\a_\l(t,x)]} - |w(T_j,x)|^2\exp{[2r_0\a_\l(T_j,x)]}\\
= &\int_{T_j}^t D_s\{|w(s,x)|^2\exp{[2r_0\a_\l(s,x)]}\}\,ds
\no\\
=&2\int_{T_j}^t w(s,x)D_sw(s,x)\exp{[2r_0\a_\l(s,x)]}\,ds\\
& + 2r_0\int_{T_j}^t |w(s,x)|^2(e^{2\lambda\|\psi\|_{\infty}}-e^{\lambda\psi(x)})l'(s)(l(s))^{-2}\exp{[2r_0\a_\l(t,x)]})\,ds,
\end{align*}
which holds true for $j=1,2$, and Young inequality we easily deduce that the following inequality holds for all
$t\in (T_1,T_2)$, $\ve \in \rsp_+$ and for $j=1,2$:
\begin{align*}
&|w(T_j,x)|^2\exp{[2r_0\a_\l(T_j,x)]}\\
 \le &|w(t,x)|^2\exp{[2r_0\a_\l(t,x)]} + \ve^2\bigg |\int_{T_j}^t |D_sw(s,x)|^2\exp{[2r_0\a_\l(s,x)]}\,ds\bigg |\\
&+\ve^{-2}\bigg |\int_{T_j}^t |w(s,x)|^2\exp{[2r_0\a_\l(s,x)]}\,ds\bigg |\\
&+ 2r_0\bigg |\int_{T_j}^t |w(s,x)|^2c_{1,\lambda}(\psi)|l'(s)|(l(s))^{-2}\exp{[2r_0\a_\l(t,x)]}\,ds\bigg |\\
\le & |w(t,x)|^2\exp{[2r_0\a_\l(t,x)]} + \ve^2\int_{T_1}^{T_2} |D_sw(s,x)|^2\exp{[2r_0\a_\l(s,x)]}\,ds\\
&+ \int_{T_1}^{T_2} |w(s,x)|^2[\ve^{-2}+2r_0c_{1,\lambda}(\psi)|l'(s)|(l(s))^{-2}]\exp{[2r_0\a_\l(s,x)]}\,ds.
\end{align*}
Integrating over $(T_1,T_2)$ the first and last side of the previous chain of inequalities yields
\begin{align*}
&|w(T_j,x)|^2\exp{[2r_0\a_\l(T_j,x)]}\notag\\
\le &\ve^2\int_{T_1}^{T_2} |D_tw(t,x)|^2\exp{[2r_0\a_\l(t,x)]}\,dt
\no\\
& +\!\int_{T_1}^{T_2} |w(t,x)|^2\{(T_2-T_1)^{-1}\!+\!\ve^{-2}\!+\!2r_0c_{1,\lambda}(\psi)|l'(t)|(l(t))^{-2}\}\exp{[2r_0\a_\l(t,x)]}\,dt.
\end{align*}
Finally, an integration over $\Omega$ leads to the assertion.
\end{proof}

%
The needed estimates for ${\mathcal B}_1$ and ${\mathcal B}_2$ follow
from \eqref{3.4b}, if we choose  $\ve=s_0^{-(1+\dl)/2}$, with $\dl\in (0,1)$, and observe that $l(t)\ge M_{T_1,T_2}:=\min\{T_1(T-T_1), T_2(T-T_2)\}$ for
any $t\in [T_1,T_2]$ and $l(t)\le 2^{-2}T^2$, $|l'(t)|\le T$ for any $t\in [0,T]$.
Indeed,
\begin{align}
&\int_{Q_T} |{\mathcal B}_jv(t,x)|^2\exp{[2s_0\a_\l(T_j,x)]}\,dtdx
\no\\
=&\int_{\Om}|v(T_j,x)|^2\exp{[2s_0\a_\l(T_j,x)]}\,dx\int_0^T |f_j(t,x)|^2\,dt\no\\
\le &\|f_j\|^2_{L^\infty(\Om;L^2(0,T))}\int_{\Om}|v(T_j,x)|^2
\exp{[2s_0\a_\l(T_j,x)]}\,dx\no\\
\le &s_0^{-(1+\dl)}\|f_j\|^2_{L^2(0,T;L^{\infty}(\Omega))}
\int_{Q_{T_1,T_2}} (l(t))^{-1}l(t)|D_tv(t,x)|^2
\exp{[2s_0\a_\l(t,x)]}\,dtdx
\no\\
& + \|f_j\|^2_{L^2(0,T;L^{\infty}(\Omega))}
\int_{Q_{T_1,T_2}} \{[(T_2-T_1)^{-1}+s_0^{1+\dl}](l(t))^{3}
+ 2s_0c_{1,\lambda}(\psi)|l'(t)|l(t)\}\no\\
& \phantom{\qq+ \|f_j\|^2_{L^\infty(\Om;L^2(0,T))}\int_{Q_{T_1,T_2}} \{ }
\times (l(t))^{-3}|v(t,x)|^2\exp{[2s_0\a_\l(t,x)]}\,dtdx
\no\\
\le &s_0^{-(1+\dl)}\|f_j\|^2_{L^2(0,T;L^{\infty}(\Omega))}M_{T_1,T_2}^{-1}
\int_{Q_{T_1,T_2}}l(t)|D_tv(t,x)|^2\exp{[2s_0\a_\l(t,x)]}\,dtdx
\no\\
&+ \|f_j\|^2_{L^2(0,T;L^{\infty}(\Omega))}
\{2^{-6}T^6[(T_2-T_1)^{-1}+s_0^{1+\dl}]
+ 2^{-2}T^3s_0c_{1,\lambda}(\psi)\}
\no\\
&\qquad\qquad
\times \int_{Q_{T_1,T_2}}(l(t))^{-3}|v(t,x)|^2\exp{[2s_0\a_\l(t,x)]}\,dtdx,
\label{3.7}
\end{align}
for $j=1,2$.

\subsection{Estimating ${\mathcal B}_3$}
Using \eqref{3.12}, \eqref{3.8} and again the condition $\|l\|_{\infty}\le 2^{-2}T^2$, we easily obtain the following chain of inequalities:
\begin{align}
&\int_{Q_T} |{\mathcal B}_3v(t,x)|^2\exp{[2s_0\a_\l(t,x)]}\,dtdx
\no\\
= &\int_{Q_T} \exp{[2s_0\a_\l(t,x)]}|f_3(t,x)|^2\bigg|
\int_{T_1}^{T_2} \rho_1(\s,x)v(\s,x)\,d\s\bigg|^2\,dtdx
\no\\
\le &(T_2-T_1)\int_{Q_T} \exp{[2s_0\a_\l(t,x)]}|f_3(t,x)|^2\,dtdx
\int_{T_1}^{T_2} |\rho_1(\s,x)|^2|v(\s,x)|^2\,d\s
\no\\
=& (T_2-T_1)\int_{Q_{T_1,T_2}} |\rho_1(\s,x)|^2|v(\s,x)|^2\,d\s dx
\int_0^T \exp{[2s_0\a_\l(t,x)]}|f_3(t,x)|^2\,dt
\no\\
\le &(T_2-T_1) \int_{Q_{T_1,T_2}} |\rho_1(\s,x)|^2 |v(\sigma,x)|^2
\,dxd\s \int_0^T \|f_3(t,\cdot)\|^2_{L^\infty(\Om)}\,dt
\no\\
\le &(T_2-T_1)\|f_3\|^2_{L^2(0,T;L^\infty(\Om))}K_2^2\int_{Q_{T_1,T_2}}
\exp{[2s_0\a_\l(\s,x)]}
|v(\s,x)|^2\,d\s dx
\no\\
\le &(T_2-T_1)\|f_3\|^2_{L^2(0,T;L^\infty(\Om))}
K_2^2\int_{Q_T} (l(\s))^3(l(\s))^{-3}\exp{[2s_0\a_\l(\s,x)]}|v(\s,x)|^2\,d\s dx\no\\
\le &(T_2-T_1)\|f_3\|^2_{L^2(0,T;L^\infty(\Om))}2^{-6}T^6K_2^2\no\\
&\quad\;\times \int_{Q_T} (l(\s))^{-3}\exp{[2s_0\a_\l(\s,x)]}|v(\s,x)|^2\,d\s dx.
\label{3.10}
\end{align}

\subsection{Estimating ${\mathcal B}_4=B$}
Via H\"older's inequality, we obtain
\begin{align}
&\int_{Q_T} |{\mathcal B}_4v(t,x)|^2\exp{[2s_0\a_\l(t,x)]}\,dtdx
\no\\
= &\int_{Q_T} \bigg|\int_\Om
k(t,x,y)v(t,y)\,dy\bigg|^2\exp{[2s_0\a_\l(t,x)]}\,dtdx
\no\\
\le & K_2\int_0^T (l(t))^{-\g}\,dt\int_{\Om} \exp{[2s_0\a_\l(t,x)]}\,dx
\int_\Om |k(t,x,y)||v(t,y)|^2\,dy
\no\\
\le & K_2 \int_{Q_T} (l(t))^{-3}|v(t,y)|^2\,dtdy
\int_{\Om} (l(t))^{3-\g}\exp{[2s_0\a_\l(t,x)]}|k(t,x,y)|\,dx,
\label{3.13}
\end{align}
$K_2$ being defined by \eqref{3.0d}.
\pn
Setting
$h_{s_0,\l}(t,x,y)=(l(t))^{3-\g}\exp{\{2s_0[\a_\l(t,x)-\a_\l(t,y)]\}}$,
we easily deduce the estimates
\begin{align}
h_{s_0,\l}(t,x,y)\le &(l(t))^{3-\g}\exp{\{2s_0[\exp{(\l\|\psi\|_\infty)}
-\exp{(\l\psi_m})](l(t))^{-1}\}}\no\\
\le& (l(t))^{3-\g}\exp{[2s_0c_{1,\lambda}(\psi)(l(t))^{-1}]},
\label{3.18}
\end{align}
if $t\in [0,T]$ and $\psi(x)>\psi(y)$, and
\begin{align}
h_{s_0,\l}(t,x,y)\le (l(t))^{3-\g},
\label{3.19}
\end{align}
if $t\in [0,T]$ and $\psi(x)\le \psi(y)$.
Then from \eqref{3.15}, \eqref{3.16}, \eqref{3.18} and \eqref{3.19},
we obtain
\begin{align}
\int_{\Om} h_{s_0,\l}(t,x,y)|k(t,x,y)|\,dx
=&\int_{\{x\in \Om:\ \psi(x)>\psi(y)\}} h_{s_0,\l}(t,x,y)|k(t,x,y)|\,dx
\no\\
&+ \int_{\{x\in \Om:\ \psi(x) \le \psi(y)\}} h_{s_0,\l}(t,x,y)|k(t,x,y)|\,dx
\no\\[1mm]
\le & (l(t))^{3-\g}\exp{[2s_0c_{1,\lambda}(\psi)(l(t))^{-1}]}\no\\
&\qquad\times
\int_{\{x\in \Om:\ \psi(x)>\psi(y)\}} |k(t,x,y)|\,dx\no\\
&+ (l(t))^{3-\g}\int_{\{x\in \Om:\ \psi(x)\le \psi(y)\}}
|k(t,x,y)|\,dx\no\\
\le & K_4+K_5,
\label{3.20}
\end{align}
for any $(t,y)\in Q_T$.
Hence, from \eqref{3.13} and \eqref{3.20} we easily deduce the estimate
\begin{align}
&\int_{Q_T} \exp{[2s_0\a_\l(t,x)]}|{\mathcal B}_4v(t,x)|^2\,dtdx
\no\\
\le & K_3(K_4+K_5)\int_{Q_T} (l(t))^{-3}|v(t,x)|^2
\exp{[2s_0\a_\l(t,x)]}\,dtdx.
\label{star}
\end{align}

\subsection{Estimating ${\mathcal B}_5$}
By the definition of ${\mathcal B}_5$ in \eqref{2.12}, estimates \eqref{3.28}, \eqref{3.8} and \eqref{star} we obtain
\begin{align}
&\int_{Q_T} |{\mathcal B}_5v(t,x)|^2\exp{[2s_0\a_\l(t,x)]}\,dtdx
\no\\
= &\int_{Q_T} \exp{[2s_0\a_\l(t,x)]}|f_4(t,x)|^2\bigg|\int_{T_1}^{T_2}
\rho_2(\s,x)Bv(\s,x)\,d\s\bigg|^2\,dtdx\no\\
\le &(T_2-T_1)\int_{Q_T} \exp{[2s_0\a_\l(t,x)]}|f_4(t,x)|^2
\,dtdx\int_{T_1}^{T_2} |\rho_2(\s,x)|^2|Bv(\s,x)|^2\,d\s\no\\
\le &(T_2-T_1)\int_{Q_T}|f_4(t,x)|^2\,dtdx\int_{T_1}^{T_2}|\rho_2(\sigma,x)|^2|Bv(\s,x)|^2\,d\s\no\\
\le &(T_2-T_1)\int_0^T\|f_4(t,\cdot)\|_{L^{\infty}(\Omega)}^2\,dt\int_{Q_{T_1,T_2}}|\rho_2(\sigma,x)|^2|Bv(\s,x)|^2\,d\s dx\no\\
\le &(T_2-T_1)K_1^2\|f_4\|^2_{L^2(0,T;L^\infty(\Om))}\int_{Q_{T_1,T_2}}
\exp{[2s_0\a_\l(\s,x)]}|Bv(\s,x)|^2\,d\s dx
\no\\
\le & (T_2-T_1)K_1^2K_3(K_4+K_5)\|f_4\|^2_{L^2(0,T;L^\infty(\Om))}
\no\\
&\qquad\quad\times \int_{Q_T} (l(\s))^{-3}\exp{[2s_0\a_\l(\s,x)]}|v(\s,x)|^2\,d\s dx.
\label{3.24}
\end{align}

We can now complete the proof of Theorem \ref{teo-1}. From \eqref{3.7}, \eqref{3.10}, \eqref{3.20} and
\eqref{3.24}, we easily deduce the following estimate:
\begin{align*}
&6C_1\int_{Q_T}\sum_{j=1}^5|{\mathcal B}_jv(t,x)|^2\exp{[2s_0\a_\l(t,x)]}\,dtdx\no\\
\le & H_0(s_0)\int_{Q_T} (l(t))^{-3}\exp{[2s_0\a_\l(t,x)]}|v(t,x)|^2\,d\s dx
\no\\
&+ H_1(s_0)\int_{Q_T}l(t)|D_tv(t,x)|^2\exp{[2s_0\a_\l(t_j,x)]}\,dtdx,\no\\
\le &\frac{1}{2}s_0^3\int_{Q_T} (l(t))^{-3}\exp{[2s_0\a_\l(t,x)]}
|v(t,x)|^2\,d\s dx\no\\
& + \frac{1}{2}s_0^{-1}e^{-\lambda\|\psi\|_{\infty}}\int_{Q_T}l(t)|D_tv(t,x)|^2\exp[2s_0\a_\l(t,x)]\,dtdx.
\end{align*}
Then from \eqref{2.21}, with $s=s_0$, we deduce the estimate
\eqref{3.25}.
Thus, we have proved Theorem \ref{teo-1}.

\section{A continuous dependence result}
\label{sect4}
\setcounter{equation}{0}

The main result of this section is the following:

\begin{theorem}
\label{teo-4.1}
Under Hypotheses $\ref{hyp1}$, $\ref{hyp-1a}$ and conditions \eqref{3.0b}, \eqref{3.0a}, the strong solution $u$ to
problem $(IP1)$ satisfies the continuous dependence
estimate
\begin{align*}
&\|u(\tau,\cdot)\|^2_{L^2(\Om)}+ 2\mu_2\int_{2\ve T}^\tau
\|\nabla_xu(t,\cdot)\|^2_{L^2(\Om)}\,dt
\no\\
\le & C(\ve)\big[\|f_0\|^2_{L^2(Q_T)} + \|g\|^2_{H^1(0,T;L^2(\Om))}+ \|A(\cdot,D)g\|^2_{L^2(Q_T)}\big],
\end{align*}
for any $\ve>0$, any $\tau\in [\ve T,T]$ and some suitable positive constant $C(\ve)$ depending on $\ve$.
\end{theorem}

\begin{remark} If $f_0=g=0$, then $v=0$ in $[2\ve T,T]\times \Om$
for all $\ve\in (0,1/2)$. This implies $u=g=0$ in $(0,T]\times \Om$.
In particular, since $u\in H^1(0,T;L^2(\Om))\hookrightarrow
C([0,T];L^2(\Om))$, we can conclude that $u=0$ in $C([0,T];L^2(\Om))$, i.e.,
{\it uniqueness} holds true for the solution to problem (IP1). Moreover,
the continuous dependence results estimate the solution in $C((0,T];L^2(\Om))
\cap L^2_{\rm loc}((0,T];H^1(\Om))$ and the data in $L^2(Q_T)\times
\big[H^1(0,T;L^2(\Om))\cap L^2(0,T;H^2(\Om))\big]$.
\end{remark}

\begin{remark}
From estimates \eqref{3.15} and \eqref{3.16} it follows that
\begin{equation}
\sup_{(t,y)\in Q_T} (l(t))^{3-\gamma}\int_{\Om} |k(t,x,y)|\,dx\le K_6=\max\{K_4,K_5\}.
\label{4.8}
\end{equation}

We will use this estimate in the proof of Theorem  \ref{teo-4.1}.
\end{remark}

In the proof of Theorem \ref{teo-4.1} we need the following lemma from \cite{BS}, which we state here
as a lemma.

\begin{lemma}[Theorem 4.9 of \cite{BS}]
\label{lem-4.4} Let $z \in C([0,T])$ and $b, k \in L^1(0,T)$ be
nonnegative functions which satisfy the integral inequality
\beqno
z(\tau) \le a + \int_0^{\tau} b(s)z(s)\,ds + \int_0^{\tau} k(s)(z(s))^{1/2}\,ds,\qq\;\,
\tau\in [0,T],
\eeqno
where $p\in (0,1)$ and $a\ge 0$ are given constants. Then, the following estimate
\begin{align*}
z(\tau) \le \exp \left( \int_0^{\tau} b(s)\,ds\right)\bigg [\sqrt{a}+ \frac{1}{2}\int_0^{\tau} k(s)\exp \left(-\frac{1}{2}\int _{0}^{s}b(\s)\,d\s \right)ds\bigg ]^2
\end{align*}
holds true for any $\tau\in [0,T]$.
\end{lemma}

\begin{proof}[Proof of Theorem $\ref{teo-4.1}$]
Let us introduce a family of functions $\s_\ve\in
W^{1,\infty}(0,T)$ ($\ve\in (0,T_1/(2T))$) such that
\beqno
\hskip-14truemm
&& 0 \le \s_\ve \le 1, \qq\;\, \s_\ve(t)=0,\ t\in [0,\ve T],
\qq\;\, \s_\ve(t)=1,\ t\in [2\ve T,T].
\eeqno
It is a simple task to show that the function $v_\ve=\sigma_{\ve}v$, where
$v$ is the solution to problem (IP2), solves the following
initial and boundary-value problem:
\beqno
(DP1)\ \left\{ \hskip-2truemm
\bear{lll}
v_\ve\in H^1(0,T;L^2(\Om))\cap L^2(0,T;H^2(\Om)), &
\\[2mm]
D_tv_\ve(t,x)-A(x,D)v_\ve(t,x) &
\\[2mm]
= Bv_\ve(t,x)  + \s'_\ve(t)v(t,x) + \s_\ve(t)\ds\sum_{j\neq 4}{\mathcal B}_jv(t,x)
\\[2mm]
+{\wtil f}_\ve(t,x), & (t,x)\in Q_T,
\\[2mm]
v_\ve(0,x)=0, & x\in \Om,
\\[2mm]
v_\ve(t,x)=0, & (t,x)\in (0,T)\times \partial\Om,
\eear \right.
\eeqno
where
${\wtil f}_\ve=\s_\ve {\wtil f}$ and the operator ${\mathcal B}_j$ ($j=1,2,3,5$) are defined in \eqref{2.12}.
Recall now that $-A(\cdot,D)$ satisfies the following estimate for all
$w\in H^2(\Om)\cap H^1_0(\Om)$:
\begin{align*}
-\int_{\Omega} w A(\cdot,D)w dx = & -\int_{\Omega}\sum_{i,j=1}^nD_{x_i}(a_{i,j}D_{x_j}w)w\,dx\\
&+\int_{\Omega} \sum_{j=1}^na_j w D_{x_j}w dx+\int_{\Omega}a_0w^2\,dx\\
=& \int_{\Omega}\sum_{i,j=1}^na_{i,j}D_{x_i}wD_{x_j}w\,dx+\int_{\Omega}\sum_{j=1}^na_j w D_{x_j}w\,dx
+\int_{\Omega}a_0w^2\,dx\\
\ge &\mu_0\|\nabla_xw\|_{L^2(\Omega)}^2-\bigg (\sum_{j=1}^n\|a_j\|_{\infty}^2\bigg )^{1/2}\int_{\Omega}|\nabla_xw||w|\,dx\\
&-\|a_0\|_{\infty}\|w\|_{L^2(\Omega)}^2\\
\ge & \mu_0\|\nabla_xw\|_{L^2(\Omega)}^2-\|a_0\|_{\infty}\|w\|_{L^2(\Omega)}^2\\
&-\frac{1}{2}\bigg (\sum_{j=1}^n\|a_j\|_{\infty}^2\bigg )^{1/2}(\ve\|\nabla_xw\|_{L^2(\Omega)}^2+\ve^{-1}\|w\|_{L^2(\Omega)}^2),
\end{align*}
where $\mu_0$ is the ellipticity constant in Hypothesis \ref{hyp1}(iii). Hence, choosing $\ve$ properly, we conclude that
\begin{eqnarray*}
-\int_{\Omega} w A(\cdot,D)w\, dx\ge \frac{\mu_0}{2}\|\nabla_xw\|_{L^2(\Omega)}^2-\mu_1\|w\|_{L^2(\Omega)}^2,
\end{eqnarray*}
for some positive constant $\mu_1$. Fix $\tau\in[\ve T,T)$. Multiplying both sides of the differential equation in $(DP1)$ by $v_{\ve}$,
integrating in $[\ve T,\tau]\times\Omega$ and taking the previous estimate into account, we get
%
\begin{align}
&\|v_\ve(\tau,\cdot)\|^2_{L^2(\Om)}
+\mu_0\int_{\ve T}^{\tau}\|\nabla_xv_\ve(t,\cdot)\|^2_{L^2(\Om)}dt
- 2\mu_1\int_{\ve T}^{\tau}\|v_\ve(t,\cdot)\|^2_{L^2(\Om)}dt\notag\\[1mm]
\le & 2\int_{\ve T}^{\tau}(Bv_\ve(t,\cdot),v_\ve(t,\cdot))_{L^2(\Om)}dt
+2\int_{\ve T}^T\s'_\ve(t)\|v_\ve(t,\cdot)\|_{L^2(\Om)}\|v(t,\cdot)\|_{L^2(\Om)}dt\notag\\
&+2\sum_{j\neq 4}\int_{\ve T}^{\tau}\sigma_{\ve}(t)\|v_{\ve}(t,\cdot)\|_{L^2(\Omega)}\|{\mathcal B}_jv(t,\cdot)\|_{L^2(\Omega)}dt\no\\
&+2\int_{\ve T}^T\|v_\ve(t,\cdot)\|_{L^2(\Om)}\|{\wtil f}_\ve(t,\cdot)\|_{L^2(\Om)}dt.
\label{4.7-b}
\end{align}

Let us estimate the terms in the right-hand side of \eqref{4.7-b}. The last one is straightforward to estimate using H\"older inequality. Hence, we focus our attention on the other terms.

According to \eqref{3.0d}, \eqref{4.8} and Holmgren's
inequality (cf., e.g., \cite[Chapter 16, Theorem 3]{LA}), we can estimate
\beqn
\label{4.9}
\|Bv_\ve(t,\cdot)\|_{L^2(\Om)} \le \sqrt{K_3K_6}(l(t))^{\gamma-3}
\|v_\ve(t,\cdot)\|_{L^2(\Om)},\qq\;\, t\in (0,T),
\eeqn
and, consequently,
\begin{align}
\int_{\ve T}^{\tau}(Bv_\ve(t,\cdot),v_\ve(t,\cdot))_{L^2(\Om)}dt
\le \sqrt{K_3K_6}\int_{0}^{\tau}\chi_{(\ve T,T)}(t)(l(t))^{\gamma-3}\|v_\ve(t,\cdot)\|_{L^2(\Om)}^2\,dt.
\label{Bve}
\end{align}

Further, using the inclusion ${\rm supp}\,\s'_\ve \sub [\ve T,2\ve T]$,
we obtain the inequality
\begin{align}
\int_{\ve T}^{\tau} |\s'_\ve(t)|^2\|v(t,\cdot)\|_{L^2(\Om)}\|v_{\ve}(t,\cdot)\|_{L^2(\Om)}\,dt
\le &\int_{\ve T}^{T} |\s'_\ve(t)|^2\|v(t,\cdot)\|_{L^2(\Om)}^2\,dt\notag\\
\le &\|\s'_\ve\|^2_{L^\infty(0,T)}\int_{\ve T}^{2\ve T} \Vert v(t,\cdot)\Vert^2_{L^2(\Om)}\,dt.
\label{v-1}
\end{align}
Now, we estimate terms in \eqref{4.7-b} containing the operators
${\wtil {\mathcal B}}_jv$, $j=1,2,3,5$.
Using the inclusion ${\rm supp}\,\s_\ve \sub [\ve T,T]$,
we have the inequalities
\begin{align}
&2\int_0^\tau |\s_\ve(t)|\|v_\ve(t,\cdot)\|_{L^2(\Om)}\|
{\mathcal B}_jv(t,\cdot)\|_{L^2(\Om)}\,dt
\no\\
\le &\int_0^\tau \|v_\ve(t,\cdot)\|^2_{L^2(\Om)}\,dt
+\int_0^\tau |\s_\ve(t)|^2\|{\mathcal B}_jv(t,\cdot)\|^2
_{L^2(\Om)}\,dt\no\\
\le &\int_0^\tau \|v_\ve(t,\cdot)\|^2_{L^2(\Om)}\,dt
+ \int_{\ve T}^{T} \|{\mathcal B}_jv(t,\cdot)\|^2_{L^2(\Om)}\,dt.
\label{4.10}
\end{align}
 From the definition of ${\mathcal B}_j$, $j=1,2$, Lemma \ref{lem-3.2}
with $s_0=0$ and $\ve=1$, we deduce
\begin{align}
\int_{\ve T}^{T} \|{\mathcal B}_jv(t,\cdot)\|^2_{L^2(\Om)}\,dt
= &\int_{\ve T}^{T} \,dt\int_{\Om} |f_j(t,x)|^2|v(T_j,x)|^2\,dx\no\\
\le &\|f_j\|^2_{L^2(0,T;L^\infty(\Om))}\int_{\Omega}|v(T_j,x)|^2\,dx\no\\
\le &\|f_j\|^2_{L^2(0,T;L^\infty(\Om))}\int_{Q_{T_1,T_2}} |D_tv(t,x)|^2\,dtdx\no\\
&+ \|f_j\|^2_{L^2(0,T;L^\infty(\Om))}\int_{Q_{T_1,T_2}}
[(T_2-T_1)^{-1}+1]|v(t,x)|^2\,dtdx.
\label{4.10-a}
\end{align}
Likewise we can estimate
\begin{align}
\int_{\ve T}^{T} \|{\mathcal B}_3v(t,\cdot)\|^2_{L^2(\Om)}\,dt
= &\int_{\ve T}^{T} \,dt\int_{\Om} |f_3(t,x)|^2\bigg|\int_{T_1}^{T_2}
\rho_1(\s,x)v(\s,x)\,d\s\bigg|^2\,dx\no\\
\le &\|f_3\|_{L^2(0,T;L^{\infty}(\Omega))}^2\no\\
&\qq\times\int_{\Om}\bigg (\int_{T_1}^{T_2}|\rho_1(\s,x)|^2d\s\bigg )\bigg (\int_{T_1}^{T_2}|v(\s,x)|^2\,d\s\bigg )\,dx
\no\\
\le &\|f_3\|_{L^2(0,T;L^{\infty}(\Omega))}^2\|\rho_1\|^2_{L^2(0,T;L^{\infty}(\Omega))}
\int_{T_1}^{T_2} \|v(\s,\cdot)\|^2_{L^2(\Om)}\,d\s
\label{4.10-b}
\end{align}
and, taking \eqref{4.9} (with $v$ replacing $v_{\ve}$) into account,
\begin{align}
\int_{\ve T}^{T} \|{\mathcal B}_5v(t,\cdot)\|^2_{L^2(\Om)}\,dt
= &\int_{\ve T}^{T}dt \int_\Om |f_4(t,x)|^2\Big|\int_{T_1}^{T_2}
\rho_2(\s,x)Bv(\s,x)\,d\s\Big|^2\,dx\no\\
\le &\|\rho_2\|^2_{L^2(T_1,T_2;L^{\infty}(\Omega))}\|f_4\|_{L^2(0,T;L^{\infty}(\Omega))}^2\int_{T_1}^{T_2}\|Bv(\s,\cdot)\|_{L^2(\Omega)}^2d\s\no\\
\le &\|\rho_2\|^2_{L^2(T_1,T_2;L^{\infty}(\Omega))}\|f_4\|_{L^2(0,T;L^{\infty}(\Omega))}^2\no\\
&\qquad\quad\times
K_3K_6\int_{T_1}^{T_2} (l(\s))^{2\gamma-3}(l(\s))^{-3}\|v(\s,\cdot)\|^2_{L^2(\Om)}\,d\s
\no\\
\le &K_3K_6\max\{M_{T_1,T_2}^{2\gamma-3},2^{6-4\gamma}T^{4\gamma-6}\}\|\rho_2\|^2_{L^2(T_1,T_2;L^{\infty}(\Omega))}\notag\\
&\qquad\quad\times \|f_4\|_{L^2(0,T;L^{\infty}(\Omega))}^2\int_{T_1}^{T_2}(l(\s))^{-3}\|v(\s,\cdot)\|^2_{L^2(\Om)}
\,d\s,
\label{4.14}
\end{align}
where we also used the estimate $\|l\|_{\infty}\le T^2/4$.
Therefore, from \eqref{4.7-b} and \eqref{Bve}-\eqref{4.14}
we get the integral inequality:
\begin{align}
&\|v_\ve(\tau,\cdot)\|^2_{L^2(\Om)} + \mu_0\int_0^\tau
\|\nabla_xv_\ve(t,\cdot)\|^2_{L^2(\Om)}\,dt\notag\\
\le &\int_0^\tau b_{\ve}(t)(t)\|v_\ve(t,\cdot)\|^2_{L^2(\Om)}\,dt
+\int_0^\tau \|{\wtil f}_\ve(t,\cdot)\|_{L^2(\Om)}\|v_\ve(t,\cdot)\|_{L^2(\Om)}\,dt
\no\\
& + J_1(f_1,f_2)\int_{T_1}^{T_2} \|D_tv(t,\cdot)\|^2_{L^2(\Om)}\,dt
+ J_2(f_1,f_2,f_3,\rho_1)\int_{T_1}^{T_2} \|v(t,\cdot)\|^2_{L^2(\Om)}\,dt
\no\\
& + J_3(f_4,\rho_2)\int_{T_1}^{T_2} (l(t))^{-3}\|v(t,\cdot)\|^2_{L^2(\Om)}\,dt
+ 2\|\s'_\ve\|^2_{L^\infty(0,T)}\int_{\ve T}^{2\ve T} \|v(t,\cdot)\|_{L^2(\Om)}
^2\,dt,
\label{4.15}
\end{align}
where we have set
\begin{align*}
&b_{\ve}(t)= 2\big[\mu_1 + 1+\sqrt{K_3K_6}\chi_{(\ve T,T)}(t)(l(t))^{3-\gamma}\big],\qq\;\,  t \in (0,T),\no\\[2mm]
& J_1(f_1,f_2)=\sum_{j=1}^2\, \|f_j\|^2_{L^2(0,T;L^\infty(\Om))},\no\\[1mm]
&J_2(f_1,f_2,f_3,\rho_1)=[(T_2-T_1)^{-1}+1]\sum_{j=1}^2\, \|f_j\|^2
_{L^2(0,T;L^\infty(\Om))}\no\\
&\phantom{J_2(f_1,f_2,f_3,\rho_1)=}
+ \|\rho_1\|^2_{L^2(0,T;L^{\infty}(\Omega))}\|f_3\|^2_{L^2(0,T;L^\infty(\Om))},
\no\\[3mm]
&J_3(f_4,\rho_2)=K_3K_6\max\{M_{T_1,T_2}^{2\gamma-3},2^{6-4\gamma}T^{4\gamma-6}\}\|\rho_2\|^2_{L^2(T_1,T_2;L^{\infty}(\Omega))}
\|f_4\|^2_{L^2(0,T;L^\infty(\Om))}.
\end{align*}
Since $\ve \in (0,T_1/(2T))$, it follows that $2\ve T<T_1$ and \eqref{4.15} implies the integral inequality
\begin{align*}
& \|v_\ve(\tau,\cdot)\|^2_{L^2(\Om)} + \mu_0\int_0^\tau
\|\nabla_xv_\ve(t,\cdot)\|^2_{L^2(\Om)}\,dt\\
\le &\int_0^\tau b_{\ve}(t)\|v_\ve(t,\cdot)\|^2_{L^2(\Om)}\,dt
+\int_0^\tau \|{\wtil f}_\ve(t,\cdot)\|_{L^2(\Om)}\|v_\ve(t,\cdot)\|
_{L^2(\Om)}\,dt\\
&+ J_1(f_1,f_2)\int_{\ve T}^{T_2} \|D_tv(t,\cdot)\|^2_{L^2(\Om)}\,dt
+ J_3(f_4,\rho_2)\int_{\ve T}^{T_2} (l(t))^{-3}\|v(t,\cdot)\|^2_{L^2(\Om)}\,dt\\
&+ J_4(\ve,f_1,f_2,f_3,\rho_1)\int_{\ve T}^{T_2} \|v(t,\cdot)\|^2_{L^2(\Om)}\,dt,
\end{align*}
where
\begin{eqnarray*}
J_4(\ve,f_1,f_2,f_3,\rho_1) = J_2(f_1,f_2,f_3,\rho_1) + 2\|\s'_\ve\|^2_{L^\infty(0,T)}.
\end{eqnarray*}
Now, from \eqref{alpha} we deduce the inequalities
\begin{align*}
& (l(t))^j\exp[2s_0\a_\l(t,x)]\ge \Big[\min_{t\in [\ve T,T_2]}l(t)\Big]^j
\exp\Big\{-2s_0c_{1,\lambda}(\psi) \Big[\min_{t\in [\ve T,T_2]}l(t)\Big]^{-1}\Big\}
\\
& \phantom{(l(t))^j\exp[2s_0\a_\l(t,x)]} =: C_{2+j}(\ve,T_2,T),
\\[2mm]
& (l(t))^{-3}\exp[2s_0\a_\l(t,x)]\ge 2^6T^{-6}\exp\Big\{-2s_0c_{1,\lambda}(\psi)
\Big[\min_{t\in [\ve T,T_2]}l(t)\Big]^{-1}\Big\}\\
&\phantom{(l(t))^{-3}\exp[2s_0\a_\l(t,x)]}=: C_4(\ve,T_2,T),
\end{align*}
for all $t\in [\ve T,2\ve T]$ and $j=0,1$.
Hence, from the Carleman type estimate \eqref{3.25}, we obtain
\begin{align}
& \int_{\ve T}^{T_2} \|D_t^jv(t,\cdot)\|_{L^2(\Om)}^2\,dt\no\\
\le & (C_{4-j}(\ve,T_2,T))^{-1}\int_{\ve T}^{T_2}(l(t))^{-3+4j}
\exp[2s_0\a_\l(t,x)]\|D_t^jv(t,\cdot)\|^2_{L^2(\Om)}\,dt\no\\
\le & 12C_1(C_{4-j}(\ve,T_2,T))^{-1}s_0^{4j-3}e^{\lambda j\|\psi\|_{\infty}}\int_{Q_T} |{\wtil f}(t,x)|^2\exp{[2s_0\a_\l(t,x)]}\,dtdx\no\\
\le & 12C_1(C_{4-j}(\ve,T_2,T))^{-1}s_0^{4j-3}e^{\lambda j\|\psi\|_{\infty}}\|{\wtil f}\|^2_{L^2(Q_T)},
\label{4.21}
\end{align}
for $j=0,1$ and
\begin{align}
&\int_{\ve T}^{T_2}(l(t))^{-3}\|v(t,\cdot)\|_{L^2(\Om)}^2\,dt
\no\\
\le & (C_2(\ve,T_2,T))^{-1}\int_{\ve T}^{T_2}(l(t))^{-3}\exp[2s_0\a_\l(t,x)]
\|v(t,\cdot)\|_{L^2(\Om)}^2\,dt
\no\\
\le & 12C_1(C_2(\ve,T_2,T))^{-1}s_0^{-3}\|{\wtil f}\|^2_{L^2(Q_T)}.
\label{4.22}
\end{align}
Finally, from \eqref{4.14}, \eqref{4.21} and \eqref{4.22} we deduce the fundamental
integro-differential inequality
\begin{align*}
z_{\ve}(\tau):=&\|v_\ve(\tau,\cdot)\|^2_{L^2(\Om)} +\mu_0\int_0^\tau
\|\nabla_xv_\ve(t,\cdot)\|^2_{L^2(\Om)}\,dt\no\\
\le & \int_0^\tau b_{\ve}(t)z_\ve(t)\,dt
+\int_0^\tau \|{\wtil f}(t,\cdot)\|_{L^2(\Om)}\chi_{(\ve T,T)}(t)(z_\ve(t))^{1/2}\,dt\no\\
&+ J_5(\ve,f_1,f_2,f_3,f_4,\rho_1,\rho_2,\lambda,\psi)\|{\wtil f}\|^2_{L^2(Q_T)},
\end{align*}
for any $\tau \in (\ve T,T)$ (and, hence, for any $\tau\in [0,T]$ since $v_{\ve}(t,\cdot)=0$ for any $t\in [0,\ve T]$), where
\begin{align*}
J_5(\ve,f_1, f_2, f_3, f_4,\rho_1,\rho_2,\lambda,\psi)
=&12C_1(C_3(\ve,T_2,T))^{-1}J_1(f_1,f_2)s_0e^{\lambda\|\psi\|_{\infty}}\\
&+ 12C_1(C_2(\ve,T_2,T))^{-1}J_3(f_4,\rho_2)s_0^{-3}\no\\
&+12C_1(C_4(\ve,T_2,T))^{-1}J_4(\ve,f_1,f_2,f_3,\rho_1)s_0^{-3}.
\end{align*}

From Lemma \ref{lem-4.4} we deduce the fundamental estimate
\begin{align*}
&\|v_\ve(\tau,\cdot)\|^2_{L^2(\Om)}+\mu_0\int_0^\tau
\|\nabla_xv_\ve(t,\cdot)\|^2_{L^2(\Om)}\,dt
\no\\
\le &\bigg [J_5(\ve,f_1,f_2,f_3,f_4,\rho_1,\rho_2,\lambda,\psi)^{1/2}\|{\wtil f}\|_{L^2(Q_T)}
\exp\left (\frac{1}{2}\int_0^\tau b_{\ve}(r)\,dr\right )\no\\
&\q+ \int_0^\tau \exp\bigg (\frac{1}{2}\int_t^\tau b_{\ve}(r)\,dr\bigg )\chi_{(\ve T,T)}(t)\|\wtil f(t,\cdot)\|_{L^2(\Omega)}\,dt\bigg ]^2,
\end{align*}
for any $\tau\in [0,T]$.
In particular, for all $\tau\in [2\ve T,T]$, we easily find the desired
estimate for $u=v+g$:
\begin{align}
&\|u(\tau,\cdot)\|^2_{L^2(\Om)}+\mu_0\int_{2\ve T}^\tau
\|\nabla_xu(t,\cdot)\|^2_{L^2(\Om)}\,dt\no\\
\le &2\|g(\tau,\cdot)\|^2_{L^2(\Om)}+ 2\mu_0\int_0^\tau
\|\nabla_xg(t,\cdot)\|^2_{L^2(\Om)}\,dt
\no\\
&+ 2\bigg [J_5(\ve,f_1,f_2,f_3,f_4,\rho_1,\rho_2,\lambda,\psi)^{1/2}\exp\left (\frac{1}{2}\|b_{\ve}\|_{L^1(0,T)}\right )\|{\wtil f}\|_{L^2(Q_T)}\no\\
& \phantom{\le\qquad}
+\exp\left (\frac{1}{2}\|b_{\ve}\|_{L^1(0,T)}\right )\int_0^T\chi_{(\ve T,T)}(t)
\|\wtil f(t,\cdot)\|_{L^2(\Omega)}\,dt\bigg ]^2,
\label{4.25}
\end{align}
for any $\ve \in (0,T_1/(2T))$.
Finally, observe that from \eqref{2.17} and \eqref{4.9} we easily deduce
the estimate
\begin{align*}
\|{\wtil f}(t,\cdot)\|_{L^2(\Om)}\le &\|f_0(t,\cdot)\|_{L^2(\Om)}+ \|D_tg(t,\cdot)\|_{L^2(\Om)}
+ \|A(\cdot,D)g(t,\cdot)\|_{L^2(\Om)}\\
&+\|{\mathcal B}g(t,\cdot)\|_{L^2(\Omega)}\\
\le & \|f_0(t,\cdot)\|_{L^2(\Om)}+ \|D_tg(t,\cdot)\|_{L^2(\Om)}
+ \|A(\cdot,D)g(t,\cdot)\|_{L^2(\Om)}\\
&+\sum_{j=1}^2\|f_j(t,\cdot)\|_{L^{\infty}(\Omega)}\|g(T_j,\cdot)\|_{L^2(\Omega)}\\
&+\|f_3(t,\cdot)\|_{L^{\infty}(\Omega)}\|\rho_1\|_{L^2(T_1,T_2,L^{\infty}(\Omega))}\|g\|_{L^2(Q_{T_1,T_2})}\\
&+ \sqrt{K_3K_6}(l(t))^{\gamma-3}\|g(t,\cdot)\|_{L^2(\Om)}\\
&+\sqrt{K_3K_6}M_{T_1,T_2}^{\gamma-3}\|f_4(t,\cdot)\|_{L^{\infty}(\Omega)}\|\rho_2\|_{L^2(T_1,T_2,L^{\infty}(\Omega))}\|g\|_{L^2(Q_T)},
\end{align*}
where, as usual, $M_{T_1,T_2}=\inf_{t\in [T_1,T_2]}l(t)$. Hence,
\begin{align*}
&\int_0^T\chi_{(\ve T,T)}(t)\|\wtil f(t,\cdot)\|_{L^2(\Omega)}dt\\
\le & \sqrt{T}\bigg (\|f_0\|_{L^2(Q_T)}\!+\!\|D_tg\|_{L^2(Q_T)}\!+\!\|A(\cdot,D)g\|_{L^2(Q_T)}\!+\!\sum_{j=1}^2\|f_j\|_{L^2(Q_T)}\|g(T_j,\cdot)\|_{L^2(\Omega)}\\
&\phantom{\sqrt{T}\bigg (}+\|f_3\|_{L^2(0,T;L^{\infty}(\Omega))}\|\rho_1\|_{L^2(T_1,T_2,L^{\infty}(\Omega))}\|g\|_{L^2(Q_{T_1,T_2})}\\
&\phantom{\sqrt{T}\bigg (}+\sqrt{K_3K_6}M_{T_1,T_2}^{\gamma-3}\|f_4\|_{L^2(0,T;L^{\infty}(\Omega))}\|\rho_2\|_{L^2(T_1,T_2,L^{\infty}(\Omega))}\|g\|_{L^2(Q_T)}\bigg )\\
&\phantom{\sqrt{T}\bigg (}+\sqrt{K_3K_6}\bigg (\int_{\ve T}^T(l(t))^{2\gamma-6}dt\bigg )^{1/2}\|g\|_{L^2(Q_T)}.
\end{align*}
Replacing this estimate in \eqref{4.25}, the assertion follows at once.
\end{proof}


\med
\pn
{\bf Acknowledgements.} The second author is a member of G.N.A.M.P.A. of
the Italian Istituto Nazionale di Alta Matematica.
The third author is partially supported by Grant-in-Aid for Scientific
Research (S) 15H05740 of Japan Society for the Promotion of Science.


\begin{thebibliography}{99}

\bibitem{BS} D. Bainov and P. Simeonov, Integral
Inequalities and Applications, Kluwer, Dordrecht, 1992.

\bibitem{BK} A. L. Bukhgeim and M.V. Klibanov, {\it Global uniqueness of
a class of multidimensional inverse problems},
Sov. Math. Dokl. 24 (1981), 244-247.

\bibitem{FI} A. V. Fursikov and O.Yu. Imanuvilov,
Controllability of Evolution Equations, Lecture Notes Series,
Seoul National Univ., 1996.

\bibitem{I} O. Yu. Imanuvilov, Controllability of parabolic equations, Sbornik Math. 186 (1995), 879-900.

\bibitem{IY} O.Yu. Imanuvilov and M. Yamamoto,
{\it Lipschitz stability in inverse parabolic problems by the Carleman estimate},
Inverse Problems 14 (1998), 1229-1245.

\bibitem{K} M.V. Klibanov, {\it Inverse problems and Carleman
estimates}, Inverse Problems 8 (1992), 575-596.

\bibitem{LA} P. Lax, Functional Analysis,
John Wiley \& Sons, New York, 2002.

\bibitem{LO} A. Lorenzi, {\it Two severely ill-posed linear parabolic problems,}
``Alexandru Myller'' Mathematical Seminar, pp. 150-169, AIP
Conference Proc. 1329, Amer. Inst. Phys. Melville, NY, 2011.

\bibitem{LL-1}
A. Lorenzi and L. Lorenzi, {\it A strongly ill-posed problem for a degenerate parabolic equation with unbounded coefficients in an unbounded domain $\Omega\times{\mathcal O}$ of $\mathbb R^{M+N}$},
Inverse Problems \textbf{29} (2013), 025007, 22 pp.

\bibitem{LL-2}
A. Lorenzi and L. Lorenzi, {\it A strongly ill-posed integrodifferential singular parabolic problem in the unit cube of $\mathbb R^n$},
Evol. Equ. Control Theory \textbf{3} (2014), 499-524.

\bibitem{Y} M. Yamamoto, {\it Carleman estimates for parabolic
equations and applications}, Inverse Problems 25 (2009),
123013 (75pp).
\end{thebibliography}
\end{document}